\documentclass[a4paper,12pt,leqno]{amsart} 
\usepackage{a4wide}
\usepackage{amssymb}
\usepackage{amsmath} 
\usepackage{amscd} 
\usepackage{amsbsy} 
\usepackage[utf8]{inputenc}
\usepackage{ulem}

\usepackage[dvips,final]{graphicx} \usepackage[dvips]{geometry}
\usepackage{color} 
\usepackage{epsfig} \usepackage{latexsym} \usepackage{pstricks}

 \def\vs{\vskip .6cm}
 \def\ra{\rightarrow} \def\nb{\nabla}

\def\G{\Gamma}   

\def\beq{\begin{equation}} \def\eeq{\end{equation}}
\def\bi{\begin{enumerate}} \def\ei{\end{enumerate}}
\def\bea{\begin{eqnarray*}} \def\eea{\end{eqnarray*}}
\def\ba{\begin{array}} \def\ea{\end{array}}

 \def\r{\end{proof}}

\def \R{\mathbb{R}} 
 \def \Z{\mathbb{Z}}  

\def\be{\begin{equation}} \def\ee{\end{equation}}

 \def\Sp{\mathrm{Sp}} 
 \def\G2{\mathrm{G}_2} \def\Spin{\mathrm{Spin}}
\def\Hol{\mathrm{Hol}}  

\def\G{{\rm G}}
\def\ng{\nabla^{\tilde g}}

\newtheorem{epr}{Proposition}[section]

\newtheorem{lem}[epr]{Lemma} 
\newtheorem{thrm}[epr]{Theorem}
\newtheorem{ath}[epr]{Theorem} 
\newtheorem{elem}[epr]{Lemma}

\newtheorem{prob}[epr]{Problem}

\theoremstyle{definition} 
\newtheorem{ede}[epr]{Definition} \newtheorem{ere}[epr]{Remark}
\newtheorem{ex}[epr]{Example}

\def\obs{\begin{ere}} \def\eobs{\end{ere}} \def\bl\begin{elem}
\def\el\end{elem} \def\bdf\begin{ede} \def\edf\end{ede}
 
\title{Weyl connections with special holonomy on compact conformal manifolds}
\author{Florin Belgun, Brice Flamencourt, Andrei Moroianu}

\address{Florin Belgun\\Institute of Mathematics “Simion Stoilow” of the Romanian Academy, 21 Calea Grivitei, 010702 Bucharest, Romania} \email{florin.belgun@imar.ro}

\address{Brice Flamencourt, UMPA - ENS Lyon, CNRS, 46 allée d’Italie, 69364 Lyon, France.}
\email{brice.flamencourt@ens-lyon.fr}

\address{Andrei Moroianu \\ Université Paris-Saclay, CNRS,  Laboratoire de mathématiques d'Orsay, 91405, Orsay, France, 
and Institute of Mathematics “Simion Stoilow” of the Romanian Academy, 21 Calea Grivitei, 010702 Bucharest, Romania}
\email{andrei.moroianu@math.cnrs.fr}


\begin{document}

\begin{abstract}
We consider compact conformal manifolds $(M,[g])$ endowed with a closed Weyl connection $\nabla$, i.e. a torsion-free connection preserving the conformal structure, which is locally but not globally the Levi-Civita connection of a metric in $[g]$. Our aim is to classify all such structures when both $\nabla$ and $\nabla^g$, the Levi-Civita connection of $g$, have special holonomy. In such a setting, $(M,[g],\nabla)$ is either flat, or irreducible, or carries a locally conformally product (LCP) structure.
Since the flat case is already completely classified, we focus on the last two cases.
When $\nabla$ has irreducible holonomy we prove that $(M,g)$ is either Vaisman, or a mapping torus of an isometry of a compact nearly Kähler or nearly parallel $\G_2$ manifold, while in the LCP case we prove that $g$ is neither Kähler nor Einstein, thus reducible by the Berger-Simons Theorem, and we obtain the local classification of such structures in terms of adapted LCP metrics.
\end{abstract}

\maketitle

\vs
\section{Introduction} A Weyl connection on a conformal manifold $(M,[g])$ is a torsion-free linear connection preserving the conformal structure $[g]$. A Weyl connection is called exact if it is the Levi-Civita connection of some metric lying in $[g]$ and closed if this property holds in the neighbourhood of each point. The analysis of exact Weyl connections thus belongs to the field of Riemannian geometry, while the closed non-exact Weyl connections fall in the area of genuine conformal geometry. In this article we will concentrate on the analysis of this latter class of connections on compact conformal manifolds.

A closed non-exact Weyl connection $\nabla$ is better understood through the study of its lift to the universal cover $\tilde M$ of $M$. Indeed, on $\tilde M$ there is a metric $h$, unique up to multiplication by a constant, belonging to the lifted conformal class and such that $\tilde \nabla$ is the Levi-Civita connection of $h$. Moreover, the deck-transformations given by $\pi_1(M)$ act as homotheties of $(\tilde M,h)$, and the assumption that $\nabla$ is not exact is equivalent to saying that $\pi_1(M)$ contains at least one non-isometric homothety. Conversely, the Levi-Civita connection of a metric $h$ on $\tilde M$ with respect to which $\pi_1(M)$ acts by homotheties which are not all isometries, descends to a closed non-exact Weyl connection on $M$.

 In \cite{ired} it is proved that on a compact conformal manifold $(M,c)$, a closed non-exact Weyl connection $ \nabla$ (or equivalently the metric $h$ on the universal cover $\tilde M$) is irreducible or flat provided that the connection is tame. This result was greatly improved successively by Matveev and Nikolayevsky \cite{MN} in the analytical case, then by Kourganoff \cite{kou} in the smooth case, who proved that only three cases can occur: the universal cover $(\tilde M, h)$ is either flat, or irreducible, or a Riemannian product between an Euclidean space and an irreducible incomplete manifold \cite[Theorem 1.5]{kou} (cf. \cite{FZ} for a more conceptual and less technical proof). This last case corresponds to $(M,[g],\nabla)$ being a locally conformally product (LCP) manifold (see Definition \ref{dlcp} below). 
 
 LCP structures have grown as a very active area of research in the last decade. In the homogeneous setting they were studied on solvmanifolds \cite{AdBM} and more generally on compact quotients of Lie groups \cite{dBM}, and their conformal vector fields have been described in \cite{FM}. A detailed presentation of LCP manifolds can be found in \cite{brice}. 

In the present work, we investigate the following:
\begin{prob}\label{problem}
Describe all compact Riemannian manifolds $(M,g)$ of dimension $n\ge 3$ whose Levi-Civita connection $\nabla^g$ has special holonomy, such that there exists a closed, non-exact Weyl connection $\nabla$ compatible with $[g]$ which also has special holonomy. 
\end{prob}
By special holonomy we mean here that the restricted holonomy groups of $\nabla$ and $\nabla^g$ are both strictly contained in $\mathrm{SO}(n)$.
Note that if $\nabla$ were exact (but different from $\nabla^g$), i.e. the Levi-Civita connection of a metric in $[g]$ non-homothetic to $g$, the above question amounts to characterize compact manifolds carrying two
non-homothetic conformally related metrics with special holonomy. This
question has been answered by the third named author \cite{m2016}.

We study separately the three possible cases, where the restricted holonomy of $\nabla$ is zero, reducible but non-zero, or irreducible. The case where $\nabla$ is flat was already classified by Fried \cite{Frie}. In every dimension $n\ge 3$, the conformal structure is induced by the quotient of $\R^n \setminus \{ 0 \}$ by the semi-direct product of a group of isometries of $\R^n$ and a homothety of ratio $\lambda > 1$, all these transformations fixing the origin. Consequently, we focus on the two remaining cases.

In the case where $(M,\nabla)$ has special {\it irreducible} holonomy, the Berger-Simons holonomy classification implies that the Riemannian manifold $(\tilde M,h)$, where $h$ is the metric induced by $\nabla$ on $\tilde M$, is either Kähler or Einstein. In the Kähler case, we can use the well-developed theory of LCK manifolds to finish the analysis, while in the Einstein case we can relate our problem to the analysis done for exact Weyl connections in \cite{m2016}. We prove in Proposition~\ref{prop2} that the only possibility is that $(M,g)$ is either Vaisman, or a mapping torus of an isometry of a compact nearly Kähler or nearly parallel $\G_2$ manifold. 

The last case, when $(M,\nabla)$ has reducible non-flat holonomy, is more difficult, and requires a deeper understanding of LCP structures. In this setting, we study separately the cases where $g$ is Kähler, Einstein or has reducible holonomy. The known examples of LCP manifolds given in \cite{brice} provide intuition concerning the results. The only examples of complex LCP manifolds constructed so far are some particular OT-manifolds, a class of manifolds introduced by Oeljeklaus and Toma \cite{OT} defined by means of algebraic number fields. When the number field has exactly two complex embeddings, the corresponding manifold admits an LCK structure, which in turn induces an LCP structure. However, these manifolds carry no Kähler metric. Consequently, there are no examples of LCP manifolds carrying a Kähler metric in their conformal class. In Theorem~\ref{prop3}, we prove that it is indeed impossible to construct such an example.

When $g$ is an Einstein metric, the conformal relation between the metric $h$ and the lift $\tilde g$ of $g$ to the universal cover $\tilde M$ of $M$ gives some link between the Ricci tensors of the two metrics. We can then exploit the particular product structure of $(\tilde M,h)$, which admits a non-trivial flat factor, in order to conclude in this case as well. We prove in Theorem~\ref{te} that no Einstein metric lies in the conformal class of an LCP manifold.

The last part of this text is devoted to the study of the case where both $h$ and $g$ have reducible holonomy. Examples of such structures are easy to construct. Indeed, every Riemannian product between a compact manifold with adapted LCP metric and another compact Riemannian manifold is again an LCP manifold (see Example~\ref{example} below or \cite[Section 3]{brice}). In Theorem~\ref{tred} below we prove that conversely, if $g$ is a reducible Riemannian metric on a compact LCP manifold $(M,[g],\nabla)$, then $g$ is adapted and $(\tilde M,\tilde g)$ is a Riemannian product $(M_1,g_1) \times (M_2,g_2)$ such that  the flat factor of $(\tilde M,h)$ is tangent to $M_1$ at each point.

At the level of the universal cover, the solutions of Problem \ref{problem} are summarized in Theorem \ref{cl} below. For a complete classification, one would need the description of discrete cocompact groups acting freely by homotheties on the simply connected Riemannian manifolds appearing in Theorem \ref{cl}, but this problem is clearly out of reach.

{\bf Acknowledgments.} This work was partly supported by the PNRR-III-C9-2023-I8 grant CF 149/31.07.2023 {\it Conformal Aspects of Geometry and Dynamics}, by the GDRI ECO-Math and by the Procope Project No. 57650868 (Germany) / 48959TL (France).

\section{Preliminaries}

\subsection{Weyl connections} Let $M$ be a compact manifold of dimension $n\ge 3$ and let $c$ be a conformal structure on $M$ (usually $c$ is the conformal class of a Riemannian metric $g$ on $M$, denoted by $c=[g]$). A {\it Weyl connection} on $M$ is a torsion-free linear connection $\nabla$ preserving the conformal structure $c$, in the sense that for every Riemannian metric $g\in c$, $\nabla_Xg=-2\theta_g(X)g$ for some $1$-form $\theta_g$ on $M$ called the {\it Lee form} of $\nb$ with respect to $g$. The Lee form of $\nb$ with respect to $g$ vanishes if and only if $\nb$ is the Levi-Civita connection of $g$, denoted by $\nb^g$.

The Weyl connection $\nb$, with the Lee form $\theta_g$ with respect to the Riemannian metric $g\in c$, acts on vector fields $Y\in C^\infty(TM)$ as follows:
\beq\label{WLC}\nb_XY=\nb^g_XY+\theta_g(Y)X+\theta_g(X)Y-g(X,Y)\theta_g^{\sharp_g},\eeq
where $\theta_g^{\sharp_g}$ is the dual vector field to the 1-form $\theta_g$ with respect to $g$.

Note that, if $g':=e^{2f}g$ is another metric in the conformal class, then 
\beq\label{tt}\theta_{g'}=\theta_g+d f.\eeq

The Weyl connection $\nb$ is called {\it closed} if $\theta_g$ is closed for one (hence for all) metrics $g\in c$ and {\it exact} if $\theta_g$ is exact (for some, hence) for all $g\in c$. From \eqref{tt}, an exact Weyl connection on $(M,c)$ is the Levi-Civita connection of a metric in the conformal class $c$.

\subsection{The metric on $\tilde M$ associated to a closed Weyl connection} Assume now that $g$ is a Riemannian metric on $M$ and $\nabla$ is a closed Weyl connection on $(M,[g])$ with Lee form $\theta_g$ with respect to $g$. We denote by $\pi:\tilde M\to M$ the universal cover of $M$ and by $\tilde g:=\pi^* g$ the induced Riemannian metric. The lift $\tilde \nb$ of $\nb$ to $\tilde M$ has Lee form $\tilde \theta:=\pi^*\theta_g$ which is exact, since $\tilde M$ is simply connected. Consequently, there exists a function $\varphi\in C^\infty (\tilde M)$, defined up to a constant, such that $\tilde\theta=d\varphi$. By \eqref{tt}, $\tilde \nb$ is the Levi-Civita connection of the metric $h:=e^{2\varphi} \tilde g$. The metric $h$ is called associated to $\nabla$, and is uniquely determined by $\nabla$ up to a multiplicative constant.

The fundamental group $\pi_1(M)$ acts on $\tilde M$ by isometries with respect to the metric $\tilde g$. Moreover, it preserves $\tilde\theta=d\varphi$, which means that $\varphi$ is a $\pi_1(M)$-{\it equivariant function}, i.e.
$$\varphi(\gamma(x))=\varphi(x)+\rho(\gamma),\ \forall\gamma\in\pi_1(M),$$
for some (nontrivial) group homomorphism $\rho:\pi_1(M)\ra\R$. We infer that
$$\gamma^*h=e^{2\rho(\gamma)}h, \ \forall \gamma\in\pi_1(M),$$
showing that $\pi_1(M)$ acts on $\tilde M$ by {\it homotheties}  with respect to the metric $h$.

In this case, \eqref{WLC} applied on $\tilde M$ becomes 
\beq\label{CH}\nb^{h}_XY=\nb^{\tilde g}_XY+d\varphi(Y)X+d\varphi(X)Y-(d\varphi)^\sharp\langle X,Y\rangle,\eeq
where we intentionally omit referring to $h$ or to $\tilde g$ in the last term, the convention being that the metric used to compute the scalar product $\langle X,Y\rangle$ is the same as the one used to ``lift indices'' $\sharp:T^*\tilde M\ra T\tilde M$.

\subsection{Holonomy issues} From now on we consider that the assumptions of Problem \ref{problem} hold. This can be equivalently stated by the fact that the restricted holonomy groups of the metrics $h$ on $\tilde M$ and $g$ on $M$ are {\it special}, that is, strictly contained in $\mathrm{SO}(n)$.

According to the classical de Rham decomposition theorem \cite[p. 185]{kn1} and the Berger-Simons holonomy classification \cite[p. 300]{besse}, there are four cases when the restricted holonomy group of a Riemannian metric (or, more generally, of a closed Weyl connection) on a $n$-dimensional manifold is special:
\begin{enumerate}
\item[(i)] the metric is irreducible and locally K\"ahler;
\item[(ii)] the metric is irreducible and Einstein with non-zero scalar curvature (either locally symmetric or the quaternionic-K\"ahler case from the Berger-Simons theorem \cite[p. 300]{besse});
\item[(iii)] the metric is irreducible and the holonomy belongs to the list 
\begin{equation}\label{berger}\mathrm{SU}(n/2)\subset\mathrm{SO}(n),\ \mathrm{Sp}(n/4)\subset\mathrm{SO}(n),\ \mathrm{G}_2\subset\mathrm{SO}(7),\ \mathrm{Spin}(7)\subset\mathrm{SO}(8)
\end{equation}
of the Berger-Simons theorem \cite[p. 300]{besse};
\item[(iv)] the metric has reducible holonomy (so is locally a product by the local de Rham theorem).
\end{enumerate}

We will discuss first the case where the Weyl connection $\nabla$ has irreducible holonomy (cases (i)--(iii) in the above list). Then, in the last section we will treat the reducible case, which is more delicate.

\section{Irreducible Weyl holonomy} 

In this section we
consider the case when $(M^n,g)$ is a compact Riemannian manifold with special holonomy and the closed non-exact Weyl connection $\nb$ has special {\it irreducible} holonomy. As explained before, the lift $\tilde\nb$ of $\nb$ to the universal cover $\tilde M$ of $M$ is the Levi-Civita connection of a metric $h$ on $\tilde M$ which belongs to the conformal class determined by the lift $\tilde g$ of $g$ to $\tilde M$. 

Because the restricted holonomy group $\Hol_0(\nb)=\Hol(\tilde M,h)$ is a Riemannian
holonomy, we need to consider the following cases:

{\bf Case (i)} $(\tilde M,h)$ is Kähler. Then $(M,g)$ is a {\it locally conformally Kähler} (LCK) compact manifold with special holonomy which is not globally conformally Kähler, since $\nabla$ is assumed to be non-exact. By Case 1. of Thm. 1.3 from \cite{fam} it follows that $(M,g)$ is a Vaisman manifold and $\nabla$ is the canonical Weyl connection of the LCK structure.

{\bf Case (ii)} The metric $h$ is irreducible and locally symmetric, or $\Hol_0(\nb)=\Sp(n/4)\Sp(1)$ for $n\ge 8$ and multiple of $4$. It is well known that in this situation $h$ is Einstein with non-zero constant scalar curvature (see \cite[Theorem 14.39]{besse} for the case of quaternionic-Kähler manifolds). However, since $h$ admits strict homotheties (which preserve the Ricci
tensor of $h$, but not its trace with respect to $h$), this case is impossible.

{\bf Case (iii)} $\Hol_0(\nb)$ belongs to the Berger list \eqref{berger}. We will not study the first two cases, since they correspond to $(\tilde M,h)$ being Kähler, when Case (i) applies. Assume now that $\Hol_0(\nb)$ is equal to $\G_2$ for $n=7$ or $\Spin(7)$ for $n=8$. Then the manifold $\tilde M$ is spin since its frame bundle reduces to the holonomy group of $h$, which is simply connected.
By a result of Wang \cite{wang}, $(M,h)$ carries a
non-trivial parallel spinor (so in particular it is Ricci-flat).

Consequently, $(M,[g],\nb)$ is {\it Einstein-Weyl}, which
by a result of Tod \cite[Prop. 2.2]{tod} implies that the Lee form $\theta_0=d\varphi_0$ of $\nabla$ with respect to the Gauduchon metric $g_0$ is $\nabla^{g_0}$-parallel. The global de Rham theorem shows that $(\tilde M,\tilde g_0)$ is isometric to $(\R\times N^{n-1},d\varphi_0^2+g_N)$ for some $n-1$-dimensional complete Riemannian manifold $(N^{n-1},g_N)$. Using the change of coordinates $r:=e^{\varphi_0}$
we deduce that the metric
$h=e^{2\varphi_0}\tilde g_0$ on $\tilde M$ is a {\it Riemannian cone}: $(\tilde M,h)=(\R_+^*\times N^{n-1},dr^2+r^2g_N)$. By \cite{baer93a}, $(N^{n-1},g_N)$ is either nearly Kähler for $n=7$, or nearly parallel $\G_2$ for $n=8$.
Moreover, the restriction of the parallel spinor of $(\tilde M,h)$ to $(N^{n-1},g_N)$ is a real Killing spinor $\psi$ with Killing constant $1/2$:
\begin{equation}\label{psi} \nabla^{g_N}_X\psi=\frac12
X\cdot\psi,\qquad\forall X\in TN.
\end{equation} 

This implies in particular that $(N^{n-1},g_N)$ is Einstein
with positive scalar curvature (thus compact by Myers' theorem), and this implies that $(M,g_0)$ is a mapping torus of an isometry of $(N^{n-1},g_N)$.

We thus end up with two conformally related metrics with special holonomy on the compact manifold $M$: the original metric $g$, and the Gauduchon metric $g_0$ of $\nabla$. We claim that $g$ is a scalar multiple of $g_0$. 

Indeed, if this is not the case, it follows from \cite[Thm. 5.1 and Thm. 6.3]{m2016} that $\R\times N$ admits a triple warped product metric, so $N$ admits a metric with reducible holonomy, which is conformal to $g_N$. We denote this metric by $g_r$. As $N$ is compact and simply connected, de Rham's decomposition theorem shows that $(N,g_r)$ is a global Riemannian product $(N_1,g_1)\times (N_2,g_2)$. By \cite[Cor. 3.4]{kr}, the conformal factor between $g_N$ and $g_r$ only depends on $N_1$ or $N_2$. Up to permuting the indices, we can therefore assume that
$g_N=e^{2f}(g_1+g_2),$ where $f:N_1\ra\R$.  

We will now exploit the relation between the $(4,0)$ curvature tensors of $g_N$ and $g_r$ (cf. \cite[Theorem 1.159]{besse}):
\begin{equation}
R^{g_N}=e^{2f}(R^{g_r}+g_r\mathbin{\bigcirc\mspace{-15mu}\wedge\mspace{3mu}}(\nabla^{g_r}df-df\otimes df+\frac12|df|_{g_r}^2g_r).
\end{equation}

Viewing the curvature applied to two vectors as a 2-form, and using the metric $g_N$ to identify vectors and 1-forms, this relation is equivalent to 
\begin{equation}
\begin{split}
R^{g_N}_{X,Y}=&R^{g_r}_{X,Y}+X^{\flat_N}\wedge\left(\nabla^{g_r}_Ydf-df(Y)df+\frac12e^{-2f}|df|_{g_r}^2Y^{\flat_N}\right)\\
&-Y^{\flat_N}\wedge\left(\nabla^{g_r}_Xdf-df(X)df+\frac12e^{-2f}|df|_{g_r}^2X^{\flat_N}\right),
\end{split}
\end{equation}
for every tangent vectors $X,Y$. Applying this relation to a vector $X\in TN_1$ and to a vector $Y\in TN_2$, and  using the fact that $R^{g_r}(X,Y)=0$, $\nabla^{g_r}_Ydf=0$ and $df(Y)=0$, we obtain 
$$R^{g_N}_{X,Y}=-Y^{\flat_N}\wedge\left(\nabla^{g_1}_Xdf-df(X)df+e^{-2f}|df|_{g_1}^2X^{\flat_N}\right).$$

On the other hand, using \eqref{psi} repeatedly we obtain
$$\frac12R^{g_N}_{X,Y}\cdot\psi=R^{g_N}_{X,Y}\psi=\frac12Y^{\flat_N}\wedge X^{\flat_N}\cdot\psi.$$

Comparing these two equations we obtain
$$-Y^{\flat_N}\wedge\left(\nabla^{g_1}_Xdf-df(X)df+e^{-2f}|df|_{g_1}^2X^{\flat_N}\right)\cdot\psi=Y^{\flat_N}\wedge X^{\flat_N}\cdot\psi,\qquad\forall X\in TN_1,\ \forall Y\in TN_2,$$
whence 
$$\nabla^{g_1}_Xdf-df(X)df+e^{-2f}|df|_{g_1}^2X^{\flat_N}=- X^{\flat_N},\qquad\forall X\in TN_1.$$

Applying this formula to $X$ yields
$(\mathrm{Hess}_{g_1}f)(X,X)=df(X)^2-|df|_{g_1}^2g_1(X,X)-e^{2f}g_r(X,X)$ for all  $X\in TN_1$. On the other hand, at
a point where $f$ attains its minimum on $M_1$, the left hand term is
non-negative, whereas the right hand term is strictly negative for $X\ne 0$, which is a contradiction. This proves that $g$ is proportional to $g_0$.

Summarizing, we have proved the following:

\begin{epr}\label{prop2}The only compact manifolds $(M,g)$ with special holonomy carrying a closed non-exact Weyl connection $\nabla$ with special irreducible holonomy are Vaisman manifolds or mapping tori of an isometry of a compact nearly Kähler or nearly parallel $\G_2$ manifold. 
\end{epr}

\section{Reducible Weyl holonomy}

We consider now the remaining case, when the holonomy of the closed non-exact
Weyl connection $\nabla$ on $(M,[g])$ (or, equivalently, of the metric $h:=e^{2\varphi}\tilde g$ on the universal cover $\tilde M$) is reducible. 

Assume first that the Weyl connection $\nabla$ is flat. This case was classified in \cite{Frie}, where it is shown that the universal cover $\tilde M$ endowed with the metric $h$ (whose Levi-Civita connection is $\tilde \nabla$) is isometric to $\R^n\setminus \{0\}$ and $\pi_1(M)$ is a semi-direct product $K \rtimes \Z$ between a finite group of isometries of $\R^n$ fixing the origin and a group generated by an homothety of ratio $\lambda < 1$. Seeing $\R^n \setminus {0}$ as the product $\R_+ \times S^{n-1}$ together with the metric $dr^2 + r^2 g_S$ where $g_S$ is the round metric on the sphere, one deduces that $\pi_1(M)$ acts by isometries for the metric $\frac{1}{r^2} dr^2 + g_S$, which is conformal to the previous metric and descends to $M$. In addition, this metric is a product metric, so it is reducible, implying that all manifolds occurring in Fried's classification \cite{Frie} are solutions of Problem \ref{problem}. In fact, the same argument as in case (iii) in the previous section shows that the product metric on $S^1\times S^{n-1}$ is the only metric with special holonomy in its conformal class, because otherwise $S^{n-1}$ would be a product of two positive-dimensional manifolds, which is clearly impossible.

We will thus assume from now on that the Weyl connection $\nabla$ is non-flat and has reducible holonomy. We introduce the following terminology.

\begin{ede} \label{dlcp}
    A Weyl connection $\nabla$ on a compact conformal manifold $(M,c)$ is called a {\it locally conformally product} (LCP) structure if it is closed, non-exact, non-flat, and has reducible holonomy. An LCP structure $(c,\nabla)$ is said to be compatible with a Riemannian metric $g$ on $M$ if $g\in c$.
\end{ede}

By assumption, the metric $h$ on the universal cover of any LCP manifold has reducible holonomy. However, as $h$ is incomplete, we cannot apply de Rham's decomposition theorem even though $\tilde M$ is simply connected. Nonetheless, we have the following:
\begin{thrm} {\rm (Kourganoff \cite[Theorem 1.5]{kou})}\label{kourg}
The universal cover $(\tilde M,h)$ of a compact LCP manifold $(M,c,\nabla)$ is globally isometric to a Riemannian product $\mathbb{R}^q\times (N, g_N)$, where $\mathbb{R}^q$ ($q\ge 1$) is the flat Euclidean space, and $(N, g_N)$ is an incomplete Riemannian manifold with irreducible holonomy.
\end{thrm}

Our assumption throughout this section is thus that $([g],\nabla)$ is an LCP structure,
and that $g$ has special holonomy. We will distinguish 3 cases,
according to the type of $g$: Kähler, irreducible
Einstein, and reducible.

In view of Theorem~\ref{kourg}, we fix some notations on LCP structures that we will use until the end of this section. For any LCP structure $([g], \nabla)$ on the compact manifold $M$ we denote by $\tilde g$ the lift of the metric $g$ to the universal cover $\tilde M$ of $M$ and by $h:=e^{2\varphi}\tilde g$ the reducible Riemannian metric on $\tilde M$ for which $\tilde \nabla=\nabla^{h}$. With respect to product decomposition $(\tilde M,h)\cong\mathbb{R}^q\times (N, g_N)$ where $\R^q$ is an Euclidean space and $(N, g_N)$ is an irreducible incomplete manifold. the tangent bundle of $\tilde M$ decomposes into the orthogonal $\nabla^h$-parallel direct sum $T \R^q \oplus TN$. We call $T \R^q$ the {\it flat distribution} and $T N$ the {\it non-flat distribution} of the LCP structure.

\subsection{LCP structures on compact Kähler manifolds} The aim of this subsection is to prove the following:
\begin{ath}\label{prop3} On compact Kähler manifolds, there are no LCP structures compatible with the Kähler metric.
\end{ath}
\begin{proof}

Assume that $(M,g,J)$ is a compact Kähler
manifold admitting an LCP structure
$([g],\nabla)$. This situation is similar to the one studied in
Section 6 of \cite{fam}, except that loc.\ cit.\ dealt with reducible metrics which are locally conformally Kähler, while here we
consider Kähler metrics which are locally conformally reducible.

On the universal cover $\tilde M$ of $M$ 
the metric $\tilde g$ is Kähler with respect to the lift $\tilde J$.
In order to use some results from \cite{fam}, we introduce the notation $D_1:=T\R^q$ and $D_2:=TN$.

Consider first the case when the dimensions $n_1$ and $n_2$ of $D_1$
and $D_2$ are both at least 2. Most arguments in Theorem 6.2 in
\cite{fam} are valid without the compactness assumption. More
precisely, the second formula on page 143 of \cite{fam} shows the following:

\begin{epr} Assume that $T\tilde M=D_1\oplus D_2$ is a $\nabla^{h}$-parallel splitting on a Riemannian manifold
$(\tilde M^n,h)$ with $n_i:=\dim(D_i)\ge 2$, and assume
moreover that a conformally related metric $\tilde g:=e^{-2\varphi}h$
on $\tilde M$ is Kähler. Then 
\begin{equation}\label{theta} 
\frac1{n_1}\left(|\tilde\theta_1|_h^2-{\delta^h\tilde\theta_1}\right)+\frac1{n_2}\left(|\tilde\theta_2|_h^2-{\delta^h\tilde\theta_2}\right)-|\tilde\theta|_h^2=0,
\end{equation}
where $\tilde\theta_i$ denotes the restriction of $\tilde\theta:=d\varphi$ to $D_i$.
\end{epr}

The conformal change formulas (cf. \cite{besse}, Theorem 1.159) give
$$|\alpha|_{h}^2=e^{-2\varphi}|\alpha|_{\tilde g}^2,\qquad \delta^{h}\alpha=e^{-2\varphi}(\delta^{\tilde g}\alpha-(n-2)\tilde g(d\varphi,\alpha))$$
for every $1$-form $\alpha$. Equation \eqref{theta} thus becomes 
\begin{equation}\label{theta1}
\frac1{n_1}\left(|\tilde\theta_1|_{\tilde g}^2-{\delta^{\tilde g}\tilde\theta_1}+(n-2)\tilde g(\tilde\theta,\tilde\theta_1)\right)+\frac1{n_2}\left(|\tilde\theta_2|_{\tilde g}^2-{\delta^{\tilde g}\tilde\theta_2}+(n-2)\tilde g(\tilde\theta,\tilde\theta_1)\right)-|\tilde\theta|_{\tilde g}^2=0.
\end{equation} As $n_1+n_2=n$, $\tilde g(\tilde\theta,\tilde\theta_i)=|\tilde\theta_i|_{\tilde g}^2$
and $|\tilde\theta|_{\tilde g}^2=|\tilde\theta_1|_{\tilde g}^2+|\tilde\theta_2|_{\tilde g}^2$, we thus get
\begin{equation}\label{theta2}
\frac1{n_1}\left((n_2-1)|\tilde\theta_1|_{\tilde g}^2-{\delta^{\tilde g}\tilde\theta_1}\right)+\frac1{n_2}\left((n_1-1)|\tilde\theta_2|_{\tilde g}^2-{\delta^{\tilde g}\tilde\theta_2}\right)=0.
\end{equation}
The forms $\tilde\theta_i$ are $\pi_1(M)$-invariant, so they are pull-backs of 1- forms $\theta_i$ defined on $M$. The relation \eqref{theta2} thus projects to the compact quotient
$M=\tilde M/\pi_1(M)$ into
\begin{equation}\label{theta2'}
\frac1{n_1}\left((n_2-1)|\theta_1|_{g}^2-{\delta^{g}\theta_1}\right)+\frac1{n_2}\left((n_1-1)|\theta_2|_{g}^2-{\delta^{g}\theta_2}\right)=0.
\end{equation}
 so after integration over $M$ we obtain
$\theta=0$, which is excluded. 

It remains to study the case when one of the distributions $D_1$ or $D_2$ has dimension 1. Since $\tilde M$ is simply connected, this distribution determines a $\nabla^h$-parallel vector
field $\zeta$ on $\tilde M$ of unit length with respect to $h$. Then the vector field $\tilde\xi:=e^{\varphi}\zeta$ has unit length with respect to $\tilde g$, and by 
\eqref{CH} (applied to $\tilde g$ instead of $g$) we get 
\begin{equation}\label{nxi}\nabla^{\tilde g}_X\tilde \xi=\nabla^h_X\tilde\xi-d\varphi(\tilde \xi)X-d\varphi(X)\tilde \xi+\tilde g(X,\tilde\xi)(d\varphi)^\sharp=-d\varphi(\tilde \xi)X+\tilde g(X,\tilde\xi)(d\varphi)^\sharp,\end{equation}
where we have used that $\nabla^h_X\tilde\xi=d\varphi(X)e^\varphi\zeta=d\varphi(X)\tilde\xi$.
Up to passing to a double cover of $M$ if necessary, $\tilde\xi$ projects to a unit length vector field
$\xi$ on $(M,g)$ which by \eqref{nxi} satisfies
\begin{equation}\label{xi} \nabla^g_X\xi=-\theta(\xi)X+g(X,\xi)\theta
\end{equation} (we identify from now on vectors and 1-forms using the
metric $g$). We decompose $\theta$ as 
$$\theta=a\xi+bJ\xi+\theta_0,$$
where $a:=\theta(\xi)$, $b:=\theta(J\xi)$ and
$\theta_0(\xi)=\theta_0(J\xi)=0$. Since $J$ is $\nabla^g$-parallel,
\eqref{xi} immediately gives
\begin{equation}\label{xi1} \nabla^g_X\xi=-aX+g(X,\xi)\theta, \qquad
\nabla^g_XJ\xi=-aJX+g(X,\xi)J\theta,\qquad\forall X\in TM.
\end{equation} 

Let us denote by $m$ the complex dimension of $M$, so that $n=2m$, $m>1$. Using a local orthonormal basis
$\{e_i\}_{i=1,\ldots,2m}$ of the tangent bundle and the relations
$d=\sum e_i\wedge\nabla^g_{e_i}$,  $\delta=-\sum
e_i\lrcorner\nabla^g_{e_i}$, $d^c=\sum Je_i\wedge\nabla^g_{e_i}$,
$\delta^c=-\sum Je_i\lrcorner\nabla^g_{e_i}$, we readily compute
\begin{equation}\label{dxi} d\xi=-\theta\wedge\xi, \qquad
d^c\xi=2a\Omega-\theta\wedge J\xi, \qquad \delta\xi=(2m-1)a,\qquad
\delta J\xi=b
\end{equation}
(here $\Omega:=g(J\cdot,\cdot)$ is the Kähler form).
The anti-commutation of $d^c$ and $\delta$ (cf. \cite{kg}) yields
\begin{eqnarray*} 0&=&d^c\delta\xi+\delta d^c\xi=(2m-1)d^c a-\sum
e_i\lrcorner(2e_i(a)\Omega-\nabla^g_{e_i}\theta\wedge
J\xi-\theta\wedge\nabla^g_{e_i}J\xi)\\ &=&(2m-1)Jda-2J da-\delta\theta
J\xi-\nabla^g_{J\xi}\theta+\nabla^g_\theta(J\xi)+(\delta(J\xi))\theta\\
&=&(2m-3)Jda-\delta\theta
J\xi-\nabla^g_{J\xi}\theta+\nabla^g_\theta(J\xi)+b\theta.\\
\end{eqnarray*} 
Since \eqref{xi1} implies $\nabla^g_\theta J\xi=0$, we obtain thus
\begin{equation}\label{xi3} (2m-3)Jda-\delta\theta J\xi-\nabla^g_{J\xi}\theta +b\theta =0.
\end{equation}
From \eqref{xi1} together with the fact that $d\theta=0$, we obtain for
every vector field $X$:
$$g(X,\nabla^g_{J\xi}\theta)=g(J\xi,\nabla^g_X\theta)=X(b)-g(\nabla^g_XJ\xi,\theta)=X(b)+ag(JX,\theta),$$ 
whence $\nabla^g_{J\xi}\theta=db-aJ\theta$. 
Equation \eqref{xi3} thus
reads 
\begin{equation}\label{z} 0=(2m-3)Jda-\delta\theta
J\xi-db+aJ\theta+b\theta.
\end{equation} We take the scalar product with $J\xi$ in \eqref{z} and
obtain
$$0=(2m-3)\xi(a)-\delta\theta-{J\xi}(b)+a^2+b^2,$$
which after an integration over $M$ and use of the divergence theorem and \eqref{dxi}
yields
\begin{eqnarray*} 0&=&\int_M
\left((2m-3)\xi(a)-{J\xi}(b)+a^2+b^2\right) d\mu_g=\int_M
\left((2m-3)a\delta \xi-b\delta(J\xi)+a^2+b^2
\right)d\mu_g\\&=&4(m-1)^2\int_M a^2d\mu_g.
\end{eqnarray*} This shows that the function $a$ vanishes identically,
and thus \eqref{z} becomes 
\begin{equation}\label{z1} 0=\delta\theta J\xi+db-b\theta.
\end{equation} From \eqref{xi1} and \eqref{z1} we get 
\begin{equation}\label{dbj}d(bJ\xi)=db\wedge
J\xi+bdJ\xi=b\theta\wedge J\xi-bJ\theta\wedge \xi\in \Omega^{(1,1)}M.
\end{equation} The global $i\partial\bar\partial$-Lemma
(cf. \cite{kg}) shows that there exists a real function $\psi$ on $M$
such that 
$$bJ\theta\wedge \xi-b\theta\wedge J\xi=i\partial\bar\partial \psi.$$
Applying the Lefschetz operator $\Lambda$ to this relation and using
the commutation relation $[\Lambda,\partial]=i\bar\partial^*$ (see
\cite{kg}, Eq. (14.15)), we get 
$$\Delta \psi=\frac12\Delta^{\bar\partial}\psi=\frac12\bar\partial^*\bar\partial \psi=-i[\Lambda,\partial]\bar\partial \psi=-b\Lambda(J\theta\wedge \xi-\theta\wedge J\xi)=0,$$
since $\Lambda(J\theta\wedge \xi)=-g(\theta,\xi)=-a=0$, and similarly
$\Lambda(\theta\wedge J\xi)=0$. It follows that $\psi$ is constant, so 
\begin{equation}\label{bj} bJ\theta\wedge \xi=b\theta\wedge J\xi.
\end{equation} 

We now remark that the 1-form $\theta_0$ is
harmonic. Indeed, $d\theta_0=d\theta-d(bJ\xi)=0$ by \eqref{dbj} and
\eqref{bj} and 
$$\delta\theta_0=\delta\theta-\delta(b J\xi)=\delta\theta-b\delta J\xi+J\xi(b)=\delta\theta-b^2+J\xi(b)=0$$
by \eqref{z1}. Since $(M,g,J)$ is compact Kähler, $J\theta_0$ must be
harmonic too, so in particular $dJ\theta_0=0$. We then compute
$$0=d(dJ\xi)=-d(J\theta\wedge\xi)=-d(J\theta_0\wedge\xi)=J\theta_0\wedge d\xi=J\theta_0\wedge \theta\wedge\xi.$$
Since $J\theta_0$, $\xi$ and $\theta=\theta_0+b J\xi$ are mutually orthogonal, this 
shows that $\theta_0$ vanishes identically, so 
\begin{equation}\label{thetab}\theta=bJ\xi.
\end{equation} 

The relations \eqref{xi1} now read 
\begin{equation}\label{xi2} \nabla^g_X\xi=bg(X,\xi)J\xi, \qquad
\nabla^g_XJ\xi=-bg(X,\xi)\xi,\qquad\forall X\in TM,
\end{equation} thus showing that the distribution spanned by $\xi$ and
$J\xi$ is $\nabla^g$-parallel. With the equalities \eqref{xi2}, one has
\begin{equation} \label{commutation}
[e^{-\varphi} \xi, J \xi] = \nabla^g_{e^{-\varphi} \xi} J \xi - \nabla^g_{J \xi} e^{-\varphi} \xi = - b e^{-\varphi} \xi + b e^{-\varphi} \xi = 0,
\end{equation}
and
\begin{align} \label{dxii}
d (e^{-\varphi} \xi) = d (J \xi) = 0.
\end{align}

We now consider the universal cover $\tilde M$ of $M$, endowed with the Kähler metric $\tilde g$. We will denote by $\eta$ and $\tilde J \eta$ the one-forms dual to $\tilde \xi$ and $\tilde J \tilde \xi$ via the metric $\tilde g$ in order to avoid confusions. The previous analysis shows that the distribution $(\tilde \xi, \tilde J \tilde \xi)$ is $\tilde g$-parallel. Since $(\Tilde M, \tilde g)$ is simply connected and complete, one can use the de Rham decomposition theorem to obtain
\begin{equation}\label{riemprod}
(\tilde M, \tilde g) = (M_0, g_0) \times (K, g_K)
\end{equation}
where $M_0$ is the integral manifold of the parallel distribution $(\tilde \xi, \tilde J \tilde \xi)$ endowed with the metric $g_0 := \eta^2 + (\tilde J \eta)^2$ and $(K, g_K, \tilde J \vert_K)$ is a Kähler manifold. The relations \eqref{xi2} give that $\tilde J \tilde \xi$ is a geodesic vector field, therefore its integral curve through any point $p \in M_0$ is the geodesic starting at $p$ with speed $\tilde J \tilde \xi_p$. Hence, the completeness of $(M_0, g_0)$ implies that the flow $\psi^J_t$ of $\tilde J \tilde \xi$ is defined for all times. In addition, the flow $\psi_s$ of the vector field $\zeta = e^{-\varphi} \tilde \xi$ is also defined for any time by definition of $\zeta$. We fix $p \in M_0$ and using \eqref{dxii} and the simple connectedness of $M_0$, we know there exist two functions $\sigma, \tau$ such that $d \sigma = e^{-\varphi} \eta$, $d \tau = \tilde J \eta$, $\sigma(p) = \tau(p) = 0$.

We have defined so far all the necessary objects to construct a diffeomorphism between $M_0$ and $\mathbb{R}^2$. Namely, we consider the maps:
\begin{align} \label{diffM0}
f_1 : \mathbb{R}^2 \to M_ 0, (s,t) \mapsto \psi_s \circ \psi^J_t (p), && f_2 : M_0 \to \mathbb{R}^2, x \mapsto (\sigma(x), \tau(x)).
\end{align}
By \eqref{commutation}, the two flow maps $\psi$ and $\psi^J$ commute, and then easy computations give
\begin{align} \label{differentialf}
d (f_1 \circ f_2) = \mathrm{id}, && f_1 \circ f_2 (p) = (0,0), && d (f_2 \circ f_1) = \mathrm{id}, && f_2 \circ f_1 (0,0) = p,
\end{align}
so $f_1$ and $f_2$ are inverse to each other. Consequently, $f_1$ is a diffeomorphism, and after computing the pull-back of $g_0$ by $f_1$, one obtains in the new coordinates:
\begin{equation} \label{metricM0}
(M_0, g_0) \simeq (\mathbb{R}^2, e^{-2 \varphi} d s^2 + d t^2).
\end{equation}
In these coordinates one has $\eta=e^{\varphi}ds$ and $\tilde J \eta =dt$.
Note that the function $\varphi$, viewed as a function on $\mathbb{R}^2$, only depends on $t$ due to the fact that $\zeta(\varphi) = \tilde\theta (\zeta) = \tilde\theta(e^{-\varphi}\tilde\xi)=0$  by \eqref{thetab}.

We can compute the scalar $S_0$ curvature of $(M_0,g_0)$ using the relations \eqref{xi2} and we obtain
\begin{equation}
S_0 = (\tilde J \tilde \xi)(\tilde b) - \tilde b^2
\end{equation}
where $\tilde b$ is the lift of $b$ to $\tilde M$. Thus the Riemannian surface $(M_0,g_0)$ is flat if and only if $(J \xi)(b) - b^2 = 0$. If this were the case, we would have $b^2 = 0$ at some point where $b$ attains its maximum or its minimum (which exist since $M$ is compact), whence $b \cong 0$, yielding that $\varphi$ is constant by equation~\eqref{thetab}. This is impossible since $\varphi$ is unbounded.

The Riemannian surface $(M_0, g_0)$ is thus irreducible. Moreover, the action of the fundamental group  of $M$ preserves the distribution spanned by $\frac{\partial}{\partial s}$ by assumption, so we deduce that $\pi_1(M)$ preserves the factor $M_0$ in the decomposition $\tilde M \simeq M_0 \times K$, because $\pi_1(M)$ acts by isometries on $(\tilde M, \tilde g)$, thus its action preserves the de Rham decomposition of $(\tilde M, \tilde g)$. Hence, $\pi_1(M)$ preserves the decomposition $M_0 \simeq \R \times \R$. We consider the restriction $H$ of $\pi_1(M)$ to the second factor $\R$ (i.e. the factor with the $t$ coordinate). Clearly $H$ preserves the 1-form $\tilde J \eta =dt$, so 
$H$ contains only translations.

The function $\varphi$ being equivariant, for any $\gamma \in H$ there exists a constant $\lambda_\gamma > 0$ such that
 \begin{equation} \label{equivphi}
 \gamma^* \varphi (t) = \varphi(t) + \ln \lambda_\gamma.
 \end{equation}

If $H$ is of rank $1$ (as an abelian group), then it is discrete and generated by a single translation. If it is of rank at least $2$, then it is dense in $\R$ (as a group of translations), and the map that associates to a translation in $H$ its constant given by equation~\eqref{equivphi} is a group homomorphism, so $\varphi$ is an affine map by continuity.

Summarizing the previous discussion, no matter whether $\varphi$ is affine or not, there always exists $\lambda \in \R_+ \setminus \{0,1\}$ such that the function $t\mapsto\varphi(t)-t \ln \lambda$ is $\pi_1(M)$-invariant, thus it can be written as the pull-back of a function $\psi$ on $M$.
We will now compute the co-differential of the 1-form $dt$ on $(\tilde M,\tilde g)$. By \eqref{riemprod} and \eqref{metricM0}, this is equal to the co-differential of $dt$ on $(\mathbb{R}^2, e^{-2 \varphi} d s^2 + d t^2).$ Consider the orthonormal frame $(e_1:=e^{\varphi(t)}\frac{\partial}{\partial s},e_2:=\frac{\partial}{\partial t})$ on $(\mathbb{R}^2, e^{-2 \varphi} d s^2 + d t^2).$ We have $[e_1,e_2]=-\varphi'(t) e_1$, so using the Koszul formula we readily obtain
\[
\nabla_{e_1}e_1=\varphi'(t)e_2,\qquad \nabla_{e_1}e_2=-\varphi'(t)e_1,\qquad \nabla_{e_2}e_1= \nabla_{e_2}e_2=0.
\]
This gives $\delta^{\tilde g}dt=\varphi'(t)=\psi'(t)+\ln\lambda$ and consequently we obtain $\delta^{g}(e^\psi dt)=e^\psi\ln\lambda$ on $(M,g)$ (we denoted here by the same symbol the projection of $dt$ to $M$).  Integrating against the volume form of $g$ yields $\ln\lambda=0$ which is absurd because $\varphi$ is unbounded. This concludes the proof.
\end{proof}

\subsection{LCP structures on compact Einstein manifolds} In this subsection we will prove the following:

\begin{ath}\label{te}
On a compact Einstein manifold, there are no LCP structures compatible with the Einstein metric.
\end{ath}

\begin{proof}
Assume that
$([g],\nabla)$ is an LCP structure on an Einstein
manifold $(M,g)$, i.e. satisfying
\begin{equation}\label{einst}\mbox{Ric}^g=\lambda g\end{equation} 
for some real constant $\lambda$.

Since $\tilde g=e^{-2\varphi}h$, \eqref{CH} shows that the connections $\nabla^{\tilde g}$ and
$\nabla^h$ are related by:
\beq\label{nbnbtl}\nabla^h_XY-\nb^{\tilde g}_XY=d\varphi(Y)X+d\varphi(X)Y- \tilde g(X,Y)(d\varphi)^{\sharp_{\tilde g}},\eeq 
where $(d\varphi)^{\sharp_{\tilde g}}$ is the dual vector field to $d\varphi$ with respect to $\tilde g$.

Moreover, from \cite[Theorem 1.159]{besse} we have the following relations between the Laplace operators on functions and 
the Ricci tensors of the metrics $\tilde g$ and $h$:
\beq\label{laplace}\Delta^{\tilde g}f=e^{2\varphi}(\Delta^hf+(n-2)h(df,d\varphi)),\eeq
and
\beq\label{ric1}\mbox{Ric}^{\tilde g}-\mbox{Ric}^{h}=(n-2)(\nabla^{h}(d\varphi)+
d\varphi\otimes d\varphi)-(\Delta^{h}\varphi+(n-2)\|d\varphi\|_{h}^2)h,\eeq

Note that $\mathrm{Ric}^h(X,Y)=0$ for every $X\in T\mathbb{R}^q$ and $Y\in TN$.
Let us fix vector fields $X\in \mathcal{X}(\mathbb{R}^q)$ and $Y\in \mathcal{X}(N)$, identified with their canonical extensions to $\tilde M$. In particular, $\nabla^h_{X}Y=0.$ Plugging $X,Y$ into \eqref{ric1} and using \eqref{einst} yields
\beq\label{xy} \nabla^{h}(d\varphi)(X,Y)+
d\varphi(X)d\varphi(Y)=0.
\eeq
This implies
$$X(Y(e^\varphi))=X(e^\varphi(d\varphi(Y)))=e^\varphi(d\varphi(X)d\varphi(Y)+\nabla^{h}(d\varphi)(X,Y))=0.$$
Since this holds for every $X\in \mathcal{X}(\mathbb{R}^q)$ and $Y\in \mathcal{X}(N)$, we see that $e^\varphi=f_1+f_2$ for some smooth functions $f_1\in C^\infty(\mathbb{R}^q)$ and $f_2\in C^\infty(N)$. 

Consider an element $\gamma\in\pi_1(M)$ acting on $\tilde M$ as a strict homothety of $h$. Since $\tilde g$ is of course $\pi_1(M)$-invariant, this means that there exists a positive real number $\mu\ne 1$ such that $\gamma^*(e^\varphi)=\mu e^\varphi$.
We thus obtain 
$$\gamma^*f_1-\mu f_1=\mu f_2-\gamma^*f_2.$$
Since this is an equality between functions on $\R^q$ and $N$ respectively, there exists a constant $c$ such that $\gamma^*f_1-\mu f_1=c$. This equation can be written as 
$$\gamma^*(f_1-\frac c{1-\mu})=\mu (f_1-\frac c{1-\mu}).$$
On the other hand,  \cite[Lemma 3.4]{brice} shows that $f_1$ is bounded on $\R^q$. The above equivariance property thus shows that $f_1-\frac c{1-\mu}$ vanishes, i.e. $f_1$ is constant. 

We have thus proved that $\varphi$ is the pull-back to $\tilde M$ of a function defined on $N$. We now plug in a non-zero vector $X$ from $T\R^q$ in the Ricci transformation formula \eqref{ric1} and obtain:
$$\lambda \tilde g(X,X)=-(\Delta^{h}\varphi+(n-2)\|d\varphi\|_{h}^2)h(X,X),$$
whence 
$$\lambda=-e^{2\varphi}(\Delta^{h}\varphi+(n-2)\|d\varphi\|_{h}^2).$$

We conclude from \eqref{laplace} that 
$$-\lambda=\Delta^{\tilde g}\varphi.$$
Recall now that $d\varphi$ is the pull-back to $\tilde M$ of the Lee form $\theta$ on $M$. The previous relation thus reads $-\lambda=\delta^{g}\theta$ on $M$, which, by integration on the compact manifold $M$, yields that $\lambda=0$ and thus $\theta$ is $g$-harmonic. The Bochner formula applied to the compact Ricci-flat manifold $(M,g)$ then shows that $\theta$ is $\nabla^g$-parallel. However, this is impossible by \cite{gal} (see also \cite[Theorem 1.6]{ired}).
\end{proof}

\subsection{LCP structures on reducible manifolds}

We start by recalling the construction in \cite{brice} of LCP structures on compact manifolds carrying a reducible metric in their conformal class. Let $(c,\nabla)$ be an LCP structure on $M$.
Recall that a metric $g\in c$ is called {\it adapted} if the Lee form of $\nabla$ with respect to $g$ vanishes on the flat distribution \cite[Definition 3.8]{brice}. This is equivalent to the fact that the function $\varphi$ on $\tilde M$ defined by $h=e^{2\varphi}\tilde g$ is constant along $\R^q$, i.e. it is the pull-back of a function on $N$. By \cite[Proposition 3.6]{brice}, every LCP structure admits adapted metrics.

\begin{ex}\label{example}
Let $g'$ be an adapted metric for an LCP structure on $M'$, and let $(\tilde M',h')=\R^q\times (N',g_{N'})$ be the decomposition of the universal cover of $(\tilde M',h':=e^{2\varphi}\tilde g')$ given by Theorem \ref{kourg}. If  $(K,g_K)$ be any compact Riemannian manifold, then the  Riemannian product $(M,g):=(M',g')\times (K,g_K)$ also carries an LCP structure. Indeed, the lift of the Riemannian metric $g$ to the universal cover $\tilde M=\tilde M'\times\tilde K$ of $M$ can be written 
$$\tilde g=\tilde g'+\tilde g_K=e^{-2\varphi}h'+\tilde g_K=e^{-2\varphi}(g_{\R^q}+g_{N'}+e^{2\varphi}\tilde g_K)$$
so $(\tilde M,e^{2\varphi}\tilde g)$ is the Riemannian product of the flat space $\R^q$ and $(N,g_N):=(N'\times\tilde K,g_{N'}+e^{2\varphi}\tilde g_K)$ (the latter being a warped product metric on $N'\times \tilde K$ since $\varphi$ is a function on $N'$).

The universal cover $\tilde M$ of $M$ admits thus a Riemannian product metric $g=g_1+g_2$ -- where $g_1:=g'$ is a metric on $M_1:=\tilde M'$, and $g_2:=\tilde g_K$ on $M_2:=\tilde K$ -- which is $\pi_1(M)=\pi_1(M')\times \pi_1(K)$-invariant. The connection $\nabla^h$ induces an LCP structure on $M$, for which the flat distribution $T\R^q$ is contained in $TM_1$ and the function $\varphi:\tilde M=\R^q\times N'\times \tilde K$ determining the conformal change from $h$ to $\tilde g$ only depends on the factor $N'$.
\end{ex}

Our aim is to prove that conversely, every reducible Riemannian manifold carrying an LCP structure is obtained locally by the above construction, or, equivalently, that the properties described in the paragraph above are satisfied on every compact reducible LCP manifold.

\begin{ath} \label{tred} Assume that $(M,g)$ is a compact reducible Riemannian manifold (thus its universal cover $(\tilde M,\tilde g)$ is isometric to a Riemannian product $(M_1,g_1)\times (M_2,g_2)$ of two complete Riemannian manifolds).  If $([g],\nabla)$ is an LCP structure on $M$, defining the $\pi_1(M)$-equivariant product metric $e^{2\varphi}\tilde g=h:=g_{\R^q}+g_N$ on $\tilde M$,
  then up to interchanging the factors $M_1$ and $M_2$, the flat distribution $T\R^q$ is contained in $TM_1$, and the conformal factor $\varphi$ is the pull-back of a function on $M_1$ which is constant along $\R^q$.

Moreover, the universal cover $\tilde M$ decomposes as a triple product $\R^q\times N'\times M_2$, where $M_1=\R^q\times N'$, and $h=g_{\R^q}+g_{N'}+e^{-2\varphi}g_2$, with $(N',g_{N'})$ incomplete.
\end{ath}

\begin{proof} 
Let $\xi$ be a non-zero $\nabla^h$-parallel vector field tangent to the flat factor $\R^q$ of $\tilde M$. One can assume that $h(\xi,\xi)=1$. Since $\nabla^h\xi=0$, $\xi$ is Killing with respect to $h$ and conformal Killing with respect to the complete Riemannian product metric $\tilde g$, which is non-flat since we excluded the Einstein case in Theorem \ref{te}. Moreover, $\xi$ is complete, therefore we can apply the following remarkable result:

  {\bf Theorem (Tashiro-Miyashita)} \cite{TM67} {\it A complete conformal Killing vector field on a complete, non-flat Riemannian product is Killing.} 

We infer that  $\xi$ is Killing with respect to $\tilde g$, which implies that $\xi(\varphi)=0$. Thus $\tilde\theta(\xi)=0$, for all $h$-parallel vector fields $\xi$ tangent to the factor $\R^q$, showing that $\tilde\theta(T\R^q)=0$, or equivalently that $g$ is adapted. We will use this fact several times in the subsequent proofs without referring to it explicitly.

We consider the open set 
$$W_\xi:=\{x\in M\ |\ \xi_1\ne 0,\  \xi_2\ne 0\},$$
where for every tangent vector $X\in T\tilde M$ we will denote by $X_1$ and $X_2$ its projections to $TM_1$ and $TM_2$.

Our goal is to show that $W_\xi=\emptyset$. This will be done in two steps: first we prove that if $W_\xi$ is non-empty, then it is dense in $\tilde M$. Next, assuming $W_\xi$ is dense, we will obtain a contradiction.

\begin{lem}\label{ld}
For every $\nabla^h$-parallel vector field $\xi$, the set $W_\xi$ is either empty or dense in $\tilde M$. 
\end{lem}

\begin{proof}
The statement is trivial for $\xi=0$, so we can assume that $h(\xi,\xi)=1$. For every $x \in \tilde M$ we denote by $\tilde M_x^\perp$ the maximal integral manifold induced by the distribution $\xi^\perp$. Since $(\tilde M,h)$ is isometric to $\R^q \times (N, g_N)$ and $\xi$ is a constant vector on $\R^q$, one has $M_x^\perp \simeq \R^{q-1} \times N$.

Suppose that $W_\xi \neq \emptyset$ and fix $x \in W_\xi$. Applying \eqref{CH} to a vector field $X \in T \tilde M_x^\perp$ yields
\begin{equation}\label{nbxi}
0 = \nabla^{\tilde g}_X \xi + \tilde \theta (X) \xi = \nabla^{\tilde g}_X (e^\varphi \xi)
\end{equation}
showing that the distribution generated by $\xi$ is parallel along $\tilde M_x^\perp$. Let $y \in \tilde M_x^\perp$. One has $\xi_y \notin T M_1$ because otherwise $\xi_x$ would be in $T M_1$ since $T M_1$ is $\tilde g$-parallel. With the same argument, one has $\xi_y \notin T M_2$, and we conclude that $y \in W_\xi$ and thus $\tilde M_x^\perp \subset W_\xi$.

It remains to understand how the decomposition of $\xi$ with respect to $T M_1$ and $T M_2$ varies in the direction of $\xi$. Let $D_x$ be the maximal integral manifold through $x$ of the distribution spanned by $\xi$. We know that $D_x \simeq \R$ and $\varphi$ is constant along $D_x$ since $d \varphi (\xi) = \tilde \theta(\xi) = 0$. This implies that $\nabla^h_\xi \tilde \theta = 0$.
We remark that the metric duals $\tilde \theta^{\sharp}$ and $\tilde \theta^{\sharp_h}$ of $\theta$ with respect to $\tilde g$ and $h$ are related by $\tilde \theta^{\sharp} = e^{2 \varphi} \tilde \theta^{\sharp_h}$, showing that $\nabla^h_\xi \tilde\theta^\sharp = \nabla^h_\xi (e^{-2 \varphi} \tilde \theta^{\sharp_h}) = e^{-2 \varphi} \nabla^h_\xi \tilde \theta^{\sharp_h} = 0$.
Applying \eqref{CH} again, we obtain
\begin{align*}
0 &= \nabla^h_\xi \xi = \nabla^{\tilde g}_\xi \xi - \tilde \theta^\sharp \tilde g (\xi, \xi) = \nabla^{\tilde g}_\xi \xi - e^{-2 \varphi} \tilde \theta^\sharp \\
0 &= \nabla^h_\xi \tilde \theta^\sharp = \nabla^{\tilde g}_\xi \tilde \theta^\sharp + \tilde \theta (\tilde \theta^\sharp)\xi = \nabla^{\tilde g}_\xi \tilde \theta^\sharp + \vert \tilde \theta \vert^2_g \xi,
\end{align*}
which can be rewritten
\begin{equation} \label{eqxitheta}
\begin{aligned}
&\nabla^{\tilde g}_\xi \xi = e^{-2 \varphi} \tilde \theta^\sharp \\
&\nabla^{\tilde g}_\xi \tilde \theta^\sharp = - \vert \tilde \theta \vert^2_g \xi.
\end{aligned}
\end{equation}
Taking the $\tilde g$-scalar product with $\tilde \theta^\sharp$ in the second equation of \eqref{eqxitheta} shows that $\tilde\theta^\sharp$ has constant norm along $D_x$, so the subspace $E:=\mathrm{span}(\xi,\theta^\sharp)$ defines a $\ng$-parallel distribution along $D_x$. We distinguish two  cases.

{\bf Case 1.} Assume that $E_x \cap T_x M_1 = \lbrace 0 \rbrace = E_x \cap T_x M_2$. Since $E\cap TM_1$ and $E\cap TM_2$ are both $\ng$-parallel along $D_x$, this implies that $E_y \cap T_y M_1 = \{ 0 \}= E_y \cap T_y M_2$ for every $y\in D_x$. In particular, $\xi$ is contained neither in $T M_1$ nor in $T M_2$ along $D_x$, so $D_x \subset W_\xi$.

{\bf Case 2.} Assume now that $E_x \cap T_x M_1 \neq \lbrace 0 \rbrace$. The case $E_x \cap T_x M_2 \neq \lbrace 0 \rbrace$ will be treated similarly. Since $x\in W_\xi$, one has $\xi_x\notin T_xM_1$, so $\tilde \theta_x^\sharp \neq 0$ and the dimension of $E$ is $2$ along $D_x$. Moreover, $E_x\cap T_xM_1$ has dimension 1, so the distribution $E \cap T M_1$ has dimension $1$ along $D_x$. This allows us to define a $\ng$-parallel vector field $X \in E \cap T M_1$ along $D_x$ satisfying $\tilde g (X, X) = 1$. 
Consider the $\ng$-parallel vector field $Y$ along $D_x$ (uniquely defined up to a sign) which is orthogonal to $X$, belongs to $E$, and satisfies $\tilde g(Y, Y) = 1$.
We take a scalar product with respect to $X$ and $Y$ in equation~\eqref{eqxitheta} and obtain
\begin{align*}
&\xi (\tilde g(\xi, X)) = e^{-2 \varphi} \tilde g(\tilde \theta^\sharp, X) && \xi (\tilde g(\tilde \theta^\sharp, X)) = - \vert \tilde \theta \vert^2_{\tilde g} \tilde g(\xi, X) \\
&\xi (\tilde g(\xi, Y)) = e^{-2 \varphi} \tilde g(\tilde \theta^\sharp, Y) && \xi (\tilde g(\tilde \theta^\sharp, Y)) = - \vert \tilde \theta \vert^2_{\tilde g} \tilde g(\xi, Y).
\end{align*}
Defining $\xi_X := \tilde g(\xi, X)$, $\xi_Y := \tilde g(\xi, Y)$, $c: = e^{-\varphi} \vert \tilde \theta \vert_{\tilde g}$ and taking a further derivative with respect to $\xi$ in the first and third equations above leads to
\begin{align*}
\xi^2 (\xi_X) = - c^2 \xi_X, &&
\xi^2 (\xi_Y) = - c^2 \xi_Y.
\end{align*}
We conclude that there exist $c_1, c_2, c_3, c_4 \in \R$ such that
\begin{align}
\xi_X (t) = c_1 \cos(c t) + c_2 \sin( c t) &&
\xi_Y(t)= c_3 \cos(c t) + c_4 \sin( c t),
\end{align}
where we parameterized $D_x$ via $t \mapsto \exp_x (t \xi)$.
Due to the initial conditions, $(c_1, c_2) \neq (0, 0)$ and $(c_3, c_4) \neq (0, 0)$, so $\xi_X$ and $\xi_Y$ are analytic non-zero functions of $t$. We now write the decomposition of $X$, $Y$ and $\xi$ according to $T M_1 \oplus T M_2$ as $X=:X_1$, $Y=:Y_1 + Y_2$ and $\xi=:\xi_1 + \xi_2$ respectively, and we obtain along $D_x$:
\begin{align*}
\xi_1 = \xi_X X_1 + \xi_Y Y_1, && \xi_2 = \xi_Y Y_2.
\end{align*}
One has $\tilde g (X_1, Y_1) = \tilde g (X_1, Y)=\tilde g (X, Y) = 0$, so for any $y \in D_x$
\begin{align*}
\xi_y \in T_y M_2 \Leftrightarrow (\xi_X X_1)_y = 0 \textrm{ and } (\xi_Y Y_1)_y = 0, && \xi_y \in T_y M_1 \Leftrightarrow (\xi_Y Y_2)_y  = 0.
\end{align*}
Note that $X=X_1$ is non-vanishing along $D_x$, and $Y_2$ is also non-vanishing since otherwise $Y$ would belong to $E\cap TM_1$ which has dimension $1$ and is spanned by $X$. Thus, a necessary condition for $\xi_y$ to belong to $T_y M_2$ is that $\xi_X(y) = 0$ and a necessary condition for $\xi_y$ to belong to $T_y M_1$ is that $\xi_Y(y) = 0$. However, since the functions $\xi_X$ and $\xi_Y$ are analytic and non-zero, these two conditions occur only on a discrete subset of $D_x$. Since in Case 1 we have $D_x\subset W_\xi$, this argument shows that $U_x:=D_x \cap W_\xi$ is dense in $D_x$ in both cases. 

Now, using the fact that $\tilde M_y^\perp \subset W_\xi$ for every $y\in W_\xi$, we conclude that $W_\xi=U_x\times \tilde M_x^\perp$ which is dense in $\tilde M$, thus proving the lemma.
\end{proof}

We will now prove that $W_\xi$ is actually empty for every $\xi$.

\begin{lem}\label{le}
    For every $\nabla^h$-parallel vector field $\xi$, the set $W_\xi$ is empty. 
\end{lem}

\begin{proof} Like before, one can assume that $h(\xi,\xi)=1$. By \eqref{CH} we then obtain $0=\ng_X\xi+\tilde\theta(X)\xi-\tilde\theta^\sharp \tilde g(X,\xi)$
for every $X\in T\tilde M$. Denoting by $\eta:=e^{\varphi}\xi$, this equation reads
\begin{equation} \label{b0}\ng_X\eta=\tilde\theta^\sharp \tilde g(X,\eta)\qquad\forall X\in T\tilde M.\end{equation}
By taking a further covariant derivative in this relation and skew-symmetrizing, we obtain 
$$R^{\tilde g}_{X,Y}\eta=\tilde g(\eta,Y)(\ng_X\tilde\theta^\sharp-\tilde\theta(X)\tilde\theta^\sharp)-\tilde g(\eta,X)(\ng_Y\tilde\theta^\sharp-\tilde\theta(Y)\tilde\theta^\sharp)\qquad\forall X,Y\in T\tilde M.$$
For $X:=X_1\in TM_1$ and $Y:=X_2\in TM_2$ and using the notation $\alpha:=-e^{-\varphi}\tilde\theta$, the above relation becomes
\begin{equation}\label{eta}
    0=\tilde g(\eta,X_2)\ng_{X_1}\alpha-\tilde g(\eta,X_1)\ng_{X_2}\alpha\qquad\forall X_1\in TM_1,\ \forall X_2\in TM_2.
\end{equation}

Let us assume $W_\xi$ is non-empty, thus dense in $\tilde M$ by Lemma \ref{ld}. The relation \eqref{eta} implies that $\nabla^{\tilde g}_X\alpha=0$ at each point of $W_\xi$, for all $X$ orthogonal to the 2-plane $P$ spanned by $\eta_1$ and $\eta_2$, the components of $\eta$ in $TM_1$ and $TM_2$ which are both non-zero. In particular
  \begin{equation}\label{normalpha}
   d\|\alpha\|_{\tilde g}^2(X)=0\qquad \forall X\perp P,
 \end{equation}
 which implies that $d\|\alpha\|_{\tilde g}^2$ belongs to $P^*$ the dual 2-plane to $P$. Take now in \eqref{eta} the scalar product with $2\alpha$. We obtain
\begin{equation}\label{etawda}(\eta^\flat\wedge d\|\alpha\|_{\tilde g}^2)(X_1,X_2)=0,\ \ \forall X_i\in TM_i,\ i=1,2.\end{equation}
    But the 2-form $\eta^\flat\wedge d\|\alpha\|_{\tilde g}^2$ is a decomposable form whose factors are both in the dual 2-plane $P^*$, and, considering the basis of $P$ defined by $X_i:=\eta_i$, $i=1,2$, the above relation shows that $\eta^\flat\wedge d\|\alpha\|_{\tilde g}^2=0$ on $W_\xi$, thus everywhere on $\tilde M$. But the factor $\eta^\flat$ is unitary and, as $\alpha$ only depends on $\varphi$ which is constant on $\R^q$, $d\|\alpha\|_{\tilde g}^2(\eta)=0$, thus $d\|\alpha\|_{\tilde g}^2\perp \eta^\flat$.

    We thus obtain $d\|\alpha\|_{\tilde g}^2=0$ on $\tilde M$, so $\|\alpha\|_{\tilde g}^2=e^{-2\varphi}\|\tilde\theta\|$ is constant on $\tilde M$. But $\|\tilde\theta\|_{\tilde g}^2=\|\theta\|_{g}^2$ is bounded, whereas $\varphi$ is unbounded on $\tilde M$. This contradiction shows that $W_\xi = \emptyset$.
\end{proof}

\begin{lem}
    The distribution $T\R^q$ is either contained in $TM_1$ or in $TM_2$.
\end{lem}
\begin{proof}
Let $x$ be any point in $\tilde M$. By Lemma \ref{le}, $T_x\R^q\subset T_xM_1\cup T_xM_2$. This clearly implies that $T_x\R^q\subset T_xM_1$ or $T_x\R^q\subset T_xM_2$. For $i=1,2$, the sets
$$C_i:=\{x\in \tilde M,\ T_x\R^q\subset T_xM_i\}$$ 
are closed and disjoint. Since $\tilde M$ is connected, one of them is equal to $\tilde M$, thus proving the lemma.
\end{proof}

We can now finish the proof of Theorem \ref{tred}.

Up to exchanging $M_1$ and $M_2$ we can assume that $T_x\R^q$ is contained in $T_xM_1$ for every $x\in \tilde M$. 
Using \eqref{CH} for some non-zero $\nabla^h$-parallel section $\xi$ of $T\R^q$, we get for every $X\in T\tilde M$ 
\begin{equation} \label{b}0=\ng_X\xi+\tilde\theta(\xi)X+\tilde\theta(X)\xi-\tilde\theta^\sharp \tilde g(X,\xi).\end{equation}
The vector field $\xi$ and  is tangent to $TM_1$ and the same holds for $\ng_X\xi$ since $TM_1$ is $\ng$-parallel.
Therefore, projecting the equation above on $TM_2$ and taking $X=\xi$ yields $\tilde\theta_2=0$, so
\begin{equation}\label{theta2=0}X_2(\varphi)=0 \mbox{ for all }X_2\in TM_2.\end{equation}

The action of the Lie group $\R^q$ on $\tilde M=M_1\times M_2$ defined by the $\nabla^h$-parallel vector fields $\xi$ from the Riemannian factor $\R^q$ is free, proper, isometric with respect to $h$ and $\tilde g$, and preserves all slices $M_1\times\{x_2\}$ , $x_2\in M_2$. Fix a point $y_2\in M_2$ and denote by $N'$ the quotient space $(M_1\times \{y_2\})/ \R^q$ with respect to the above free proper action of $\R^q$. The projection $p:M_1\times \{y_2\}\ra N'$ is an $\R^q$ principal bundle with an $\R^q$-invariant horiontal (integrable) distribution $D$. As $M_1$ is simply connected, $N'$ is as well, thus the flat $\R^q$-connection $D$ in the above principal bundle has no monodromy and is therefore trivial, which makes every leaf of the foliation tangent to $D$ inside $M_1\times \{y_2\}$ diffeomorphic to the base $N'$. In fact they are also isometric when considered the metric induced by $h$, thus $(M_1\times\{y_2\},h \vert_{M_1\times\{y_2\}})$ is isometric to the Riemannian product $(\R^q,g_{\R^q})\times (N',g_{N'})$.
  
In order to prove that $\tilde M=M_1\times M_2$ is a triple product, we need to show that the decomposition $M_1\times \{y_2\}=\R^q\times N'$ is independent of the choice of $y_2\in M_2$. This in turn holds if and only if the distributions $\R^q$ and $D$ are invariant by the infinitesimal action of any vertical vector field $\bar X_2\in\mathcal{X}(\tilde M)$ with respect to the projection $p_2:\tilde M\ra M_2$, i.e. for example for $\bar X_2$ induced by a vector field $X_2\in\mathcal{X}(M_2)$. In fact $[X_1,\bar X_2]$ is a section of $TN$ (as both $X_1$ and $\bar X_2$ are) for any vector field tangent to $D\subset TN$. That makes $D\subset TM_1$ automatically stable along $\bar X_2$.

  Consider now $\xi$ a $\nabla^h$-parallel vector field tangent to $\R^q$ and compute
  $$[\xi,\bar X_2]=\nabla^{\tilde g}_{\xi}\bar X_2-\nabla^{\tilde g}_{\bar X_2}\xi.$$

The first term belongs to $TM_2$ (in fact it is zero for our choice of $\bar X_2$ being a lift of $X_2\in\mathcal{X}(M_2)$) and the second vanishes from \eqref{nbxi} since $\bar X_2\perp \xi$ and $\theta(\bar X_2)=0$ from \eqref{theta2=0}. This implies that the quotient space $M_1$ of $\tilde M$ inherits the product structure defined by the pair of integrable distributions $\R^q$ and $D$ on each slice $M_1\times\{y_2\}$, therefore $\tilde M$ is diffeomorphic to the triple product $\R^q\times N'\times M_2$. From a metric viewpoint, we have $h=g_{\R^q}+g_{N'}+e^{-2\varphi}g_2$ and $\tilde g=e^{2\varphi}(g_{\R^q}+g_{N'})+g_2$ as claimed.

In particular, $(N,g_N)$ is isometric to $(N' \times M_2, g_{N'} + e^{-2 \varphi} g_2)$, so $(N', g_{N'})$ is not complete because otherwise $(N, g_N)$ would be complete as a warped product of two complete Riemannian manifolds.
\end{proof}

Although the universal cover of $M$ turns out to be a triple product as in Example \ref{example} of a Riemannian product of an LCP manifold $M'$ with a compact manifold $K$, the next example shows that the fundamental group of $M$ is not necessarily a product of two groups acting separately on the factors $M_1$ and $M_2$, so the reducible metric $g$ on the LCP manifold $M$ is not globally a product in general.

\begin{ex}
Let $(a,b)$ be the canonical coordinate system of $\R^2$. We consider the transformation of $\R^2$ given by the matrix
\[
A = \left(
\begin{matrix}
1 & 1 \\
1 & 2
\end{matrix}
\right)
\]
and we chose a basis of eigenvectors of $A$ with associated eigenvalues $(\lambda, \lambda^{-1})$, inducing a coordinate system $(x,y)$ in $\R^2$. We define $\tilde M = \R^2 \times \R^*_+ \times \R$, and we endow this manifold with the metric
\begin{equation}
h := d x^2 + t^4 d y^2 + dt^2 + t^2 d s^2,
\end{equation}
written in the coordinate system $(x,y,t,s)$.
We now define the group of transformations $G$ generated by the maps
\begin{align*}
& (a,b,t,s) \mapsto (a+1,b,t,s) \\
& (a,b,t,s) \mapsto (a,b+1,t,s) \\
& (a,b,t,s) \mapsto (a,b,t,s + \sqrt{2}) \\
& (a,b,t,s) \mapsto (A (a,b)^T, \lambda t,s+1),
\end{align*}
written in the coordinate system $(a,b,t,s)$. It is easy to check that $G$ acts freely, properly and co-compactly on $\tilde M$ by homotheties, and the last map is an homothety of ratio $\lambda \neq 1$, so it is a strict homothety. Thus the metric $h$ defines an LCP structure on $M := \tilde M/G$ (see \cite[Remark 2.6]{brice} for more details). In addition, the metric
\begin{equation}
\tilde g := t^{-2} h = t^{-2} d x^2 + t^2 d y^2 + t^{-2} dt^2 + d s^2
\end{equation}
descends to a reducible metric $g$ on $M$. With the notations of this section, we can write
\begin{align}
(M_1,g_1) = (\R^2 \times \R^*_+, t^{-2} d x^2 + t^2 d y^2 + t^{-2} dt^2), && (M_2,g_2) = (\R ds^2),
\end{align}
but the group $G = \pi_1(M)$ is not a product of two groups acting separately on $M_1$ and $M_2$.
\end{ex}

The above example shows that our results do not answer completely Problem \ref{problem}, since we cannot describe the structure of the fundamental groups of the solutions. However, we do have a complete classification at the level of the universal covers. To make this precise, note that there is a one-to-one correspondence between the solutions $(M,g,\nabla)$ of Problem \ref{problem}, and tuples $(\tilde M,\tilde g,\varphi,\Gamma)$, where:
\begin{enumerate}
    \item $(\tilde M,\tilde g)$ is a complete simply connected Riemannian manifold of dimension $n\ge 3$ with special holonomy;
    \item $\varphi$ is a smooth function on $\tilde M$ such that the metric $h:=e^{2\varphi}\tilde g$ has special holonomy;
    \item $\Gamma$ is a discrete co-compact group acting on $\tilde M$ by isometries of $\tilde g$ and homotheties of $h$, not all of them being isometries.
\end{enumerate}

Summarizing the results in Proposition \ref{prop2}, Theorem \ref{prop3}, Theorem \ref{te} and Theorem \ref{tred}, we obtain the following classification result:

\begin{ath}\label{cl}
    The triples $(\tilde M,\tilde g,\varphi)$ satisfying conditions $(1)-(2)$ above, for which there exists a group $\Gamma$ satisfying condition $(3)$, are of the following form:
    \begin{itemize}
        \item $\tilde M=\R\times S$, $\tilde g=dt^2+g_S$, and $\varphi=dt$, where $(S,g_S)$ is a either a complete Sasakian manifold, a round sphere, or a compact nearly Kähler or nearly parallel $\G_2$ manifold.
        \item $\tilde M=\R^q\times N'\times M_2$, $\tilde g=e^{-2\varphi}(g_{\R^q}+g_{N'})+g_2$, with $q\ge 1$,  $\varphi\in C^\infty(N')$, where $(M_2,g_2)$ is a complete Riemannian manifold and $(N',g_{N'})$ is incomplete.
   \end{itemize}
 \end{ath}


\begin{thebibliography}{22} {

\bibitem{AdBM} {\sc A. Andrada, V. del Barco, A. Moroianu}, {\sl Locally conformally product structures on solvmanifolds}. Ann. Mat. Pura Appl. {\bf 203}, 2425--2456 (2024).

\bibitem{baer93a} {\sc C.~B{\"a}r}, {\sl Real {K}illing spinors and
holonomy}, Commun.\ Math.\ Phys.\ \textbf{154}, 509--521 (1993).

\bibitem{f} {\sc F.\ Belgun}, {\sl On the structure of non-K\"ahler
complex surfaces}, Math.\ Ann.\ {\bf 317}, 1--40 (2000). 

\bibitem{ired} {\sc F. Belgun, A. Moroianu}, {\sl On the irreducibility of locally metric connections}, J. reine angew. Math. {\bf 714}, 123--150 (2016).

\bibitem{af} {\sc F.\ Belgun, A.\ Moroianu}, {\sl Weyl-parallel forms,
conformal products and Einstein-Weyl manifolds}, Asian J.\ Math.\ {\bf 15}, 499--520 (2011).

\bibitem{besse} {\sc A.\ Besse,} {\it Einstein manifolds}, Ergebnisse
der Mathematik und ihrer Grenzgebiete (3)  10.\ Springer-Verlag,
Berlin, 1987.

\bibitem{dBM} {\sc V. del Barco, A. Moroianu}, {\sl The structure of locally conformally product Lie algebras}. arxiv:2404.17956 (2024).

\bibitem{brice} {\sc B.\ Flamencourt}, {\sl Locally conformally product structures}, Internat. J. Math. {\bf 35} (5), article 2450013 (2024).

\bibitem{FM} {\sc B.\ Flamencourt, A.\ Moroianu}, {\sl Conformal vector fields on LCP manifolds }. arxiv:2412.17060 (2024).

\bibitem{FZ} {\sc B.\ Flamencourt, A.\ Zeghib}, {\sl On foliations admitting a transverse similarity structure}. arxiv:2501.04814 (2025).

\bibitem{Frie}
{\sc D. Fried}, {\sl Closed similarity manifolds}. Comment. Math. Helv. {\bf 55} (4), 576--582 (1980).

\bibitem{fkms} {\sc Th.\ Friedrich, I.\ Kath, A.\ Moroianu, U.\
Semmelmann}, {\sl On nearly-parallel $\mathrm{G}_2$-structures,} J.\
Geom.\ Phys.\ {\bf 23}, 269--286 (1997).

\bibitem{gal} {\sc S.\ Gallot}, {\sl {\'E}quations diff{\'e}rentielles
caract{\'e}ristiques de la sph{\`e}re, } Ann.\ Sci.\ Ec.\ Norm.\ Sup.\
Paris {\bf 12}, 235--267 (1979). 

\bibitem{g1} {\sc P.\ Gauduchon}, {\sl La 1-forme de torsion d'une
vari\'et\'e hermitienne compacte}, Math.\ Ann.\ {\bf 267},
495--518 (1984).

\bibitem{g} {\sc P.\ Gauduchon}, {\sl Structures de Weyl-Einstein,
espaces de twisteurs et vari\'et\'es de type $S^1\times S^3$}, J.\
reine angew.\ Math.\  {\bf 469}, 1--50 (1995).  

\bibitem{gray} \textsc{A.\ Gray}, {\sl The structure of nearly
K\"ahler manifolds}, Math.\ Ann.\ {\bf 223}, 233--248 (1976).

\bibitem{kn1} {\sc S.\ Kobayashi, K.\ Nomizu}, {\it Foundations of
Differential  Geometry I}, New York, Interscience Publishers, 1963.

\bibitem{kou} {\sc M. Kourganoff}, {\sl Similarity structures and de Rham
decomposition}. Math. Ann. {\bf 373}, 1075--1101 (2019).

\bibitem{kr} {\sc W. K\"uhnel, H.-B. Rademacher}, {\sl Conformally Einstein
product spaces}, Differential Geom. Appl. {\bf 49}, 65--96 (2016).

\bibitem{lebrun} {\sc C.\ R.\ LeBrun}, {\sl Anti-self-dual Hermitian
metrics  on blown-up Hopf surfaces}, Math.\ Ann.\ {\bf 389},
383--392 (1991).
  
\bibitem {fam} {\sc F. Madani, A. Moroianu, M. Pilca}, {\sl Conformally related K\"ahler metrics and the holonomy of lcK manifolds}. J. Eur. Math. Soc. {\bf 22} (1), 119--149 (2020).

\bibitem{MN}
{\sc V.S. Matveev, Y. Nikolayevsky}, {\sl Locally conformally Berwald manifolds and compact quotients of reducible manifolds by homotheties}. Annales de l'Institut Fourier {\bf 67} (2), 843--862 (2017). 

\bibitem{schwach} {\sc S.\ Merkulov, L.\ Schwachh\"ofer}, {\sl
Classification of irreducible holonomies of torsion-free affine
connections}, Ann.\ of Math.\ {\bf 150}, 77--149 (1999).
  
\bibitem{m2016} {\sc A. Moroianu}, {\sl Conformally related metrics with
non-generic holonomy}. J. reine angew. Math. {\bf 755}, 279--292 (2019).

\bibitem{kg} {\sc A.~Moroianu}, {\it Lectures on {K}\"ahler geometry},
London Mathematical Society Student Texts, vol.~69, Cambridge
University Press, Cambridge, 2007.

\bibitem{al} {\sc A.\ Moroianu, L.\ Ornea}, {\sl Conformally Einstein
products  and nearly K\"ahler manifolds}, Ann.\ Global Anal.\ Geom.\
{\bf 33}, 11--18 (2008).

\bibitem{OT}
{\sc K. Oeljeklaus, M. Toma}, {\sl Non-Kähler compact complex manifolds associated to number fields}. Ann. Inst. Fourier {\bf 55}, 161--171 (2005).

\bibitem{TM67} {\sc Y.\ Tashiro, K.\ Miyashita}, {\sl Conformal transformations in complete product Riemannian manifolds}, 
J.\ Math.\ Soc.\ Japan {\bf 19}, 328--346 (1967).

\bibitem{tod} {\sc K.P.\ Tod}, {\sl Compact 3-dimensional Einstein-Weyl structures}, J. London Math. Soc. (2) {\bf 45}, 341--351 (1992).

\bibitem{tric} {\sc F.\ Tricceri}, {\sl Some examples of locally
conformal K\"ahler manifolds}, Rend.\ Semin.\ Mat.\ Univ.\ Politecn.\
Torino {\bf 40}, 81--92 (1982).

\bibitem{v} {\sc I.\ Vaisman}, {\sl Locally conformal K\"ahler
manifolds with parallel Lee form},  Rend.\ Mat.\ (6)  {\bf 12},
263--284 (1979).  
  
\bibitem{wang} {\sc M. Wang}, {\sl Parallel Spinors and Parallel Forms, } 
Ann Global Anal. Geom. {\bf 7}, 59--68 (1989).

\bibitem{weyl} {\sc H.\ Weyl,}  {\it Raum. Zeit. Materie, } (German),
Heidelberger Taschenb\"ucher, {\bf 251}, Springer-Verlag, Berlin,
1988.  }
\end{thebibliography}
\end{document}